\documentclass[12pt,UTF8]{amsart}
\usepackage{graphicx}
\usepackage{amsmath,amscd,amssymb,amsfonts,latexsym,wasysym,mathrsfs,mathtools,hhline,color}
\usepackage{geometry}
\usepackage[pagebackref=true,colorlinks=true, linkcolor=cyan,citecolor=cyan,urlcolor=cyan]{hyperref}

\usepackage{enumerate,nicematrix}
\usepackage{tikz}

\newtheorem{theorem}{Theorem}[section]
\newtheorem{lemma}[theorem]{Lemma}
\newtheorem{prop}[theorem]{Proposition}
\newtheorem{cor}[theorem]{Corollary}

\theoremstyle{definition}
\newtheorem{defn}[theorem]{Definition}
\newtheorem{que}[theorem]{Question}
\newtheorem{remark}[theorem]{Remark}

\newtheorem{conjecture}[theorem]{Conjecture}
\newtheorem{ex}[theorem]{Example}
\newtheorem*{ex*}{Example}

\newcommand{\R}{\mathbb{R}}

\newcommand{\C}{\mathbb{C}}
\newcommand{\CP}{\mathbb{CP}}
\newcommand{\Z}{\mathbb{Z}}
\newcommand{\Hom}{\mathrm{Hom}}

\newcommand{\K}{\mathbb{K}}

\newcommand{\sL}{\mathcal{L}}
\newcommand{\sA}{\mathcal{A}}
\newcommand{\sF}{\mathcal{F}}
\newcommand{\sC}{\mathcal{C}}

\newcommand{\rk}{\mathrm{rank}}
\newcommand{\sep}{\mathrm{Sep}}
\newcommand{\ch}{\mathrm{ch}}
\newcommand{\bch}{\mathrm{bch}}
\newcommand{\uch}{\mathrm{uch}}
\newcommand{\gal}{\mathrm{Gal}}
\newcommand{\im}{\mathrm{Im }}

\title{The cohomology groups of finite cyclic covers of complexified real arrangement complements}

\author{Wentao Xie }
\address{College of Mathematics and Computer Science, Gannan Normal University, Ganzhou 341000, China}
\email{xwt@mail.ustc.edu.cn}
\date{\today}

\begin{document}

\maketitle

\begin{abstract}
In this paper, we study the cohomology groups of finite cyclic covers of complexified real arrangement complements. 
An open problem is whether the torsion in the (co)homology of finite covering spaces of hyperplane arrangement complements, including the classical Milnor fiber, is combinatorially determined. 
Using the chamber cochain complex constructed by Yoshinaga, 
we obtain explicit upper bounds for the Betti numbers of these covers over arbitrary fields. 
Furthermore, we introduce a combinatorial condition analogous to the Cohen-Dimca-Orlik (CDO) condition. 
Under this condition, we prove that the integral cohomology groups of the finite cyclic covers are torsion-free.
This partially generalizes the recent results on complex line arrangements obtained by the author and Liu \cite[Theorem 1.3]{LX26}.  

\end{abstract}

\section{Introduction}
An affine hyperplane arrangement $\sA$ is a finite set of hyperplanes in an complex affine  space $\C^l$. 
The intersection lattice of a hyperplane arrangement is the poset of intersections of hyperplanes.  
There are several viewpoints to study hyperplane arrangements, such as topology and combinatorics. 
An important question in hyperplane arrangements is whether the topological invariants of the complement of arrangement are determined by combinatorial data (i.e., the intersection lattice). 
Classical results by Orlik and Solomon \cite{OS80} state that the homology groups and the cohomology ring of the complement are completely determined by the intersection lattice. 
Furthermore, the (co)homology of the complement is torsion-free due to its minimal CW complex structure \cite{DP03, Ran02}. However, Rybnikov \cite{Ryb11} proved that the fundamental group of the complement is not combinatorially determined.

Another open question is whether the (co)homology groups and the cohomology ring of a finite cyclic cover of a hyperplane arrangement complement are combinatorially determined.
This includes the Milnor fiber of a hyperplane arrangement, which is a covering space of the complement of $\sA$. 
Cohen--Dimca--Suciu \cite[Theorem 1.1]{CDS03} gave examples of multiarrangements whose Milnor fiber contains torsion in its first homology group. 
Denham--Suciu \cite[Theorem 1.1]{DS14} proved that for every prime $p$, there is a hyperplane arrangement whose Milnor fiber has non-trivial $p$-torsion in a higher-degree homology group. 
Yoshinaga \cite[Theorem 1.2]{Yos20} gave the first example of a hyperplane arrangement
whose Milnor fiber contains 2-torsion in its first homology group. 
These results show that the (co)homology groups of Milnor fiber may have torsion and provide a negative answer to the question of whether the (co)homology groups of the Milnor fiber of a hyperplane arrangement are torsion-free \cite[Question 7.3]{DN04}, \cite[Problem 7]{Ran11}. 
Denham--Suciu raised the following question:
\begin{que}{\cite[Question 1.2]{DS14}}
\label{Milnorfibertor}
Is the torsion in the (co)homology of the Milnor fiber of a hyperplane arrangement combinatorially determined?
\end{que}

This question naturally extends to finite cyclic covers.
In fact, the cohomology groups of finite cyclic covers of the complement are intimately related to local system cohomology. 
Firstly, we recall the definition of a general rank one local system on the complement of arrangement.
Assume that $\sA=\{H_1,\dots,H_n\}$ and $\overline{H}_\infty$ is the hyperplane at infinity.
Let $U$ be the complement of the arrangement $\sA$. 
Given a commutative ring $R$ with identity, 
Note that a rank one $R$-local system $\sL$ on $U$ corresponds to a representation of the fundamental group: $\rho_\sL:\pi_1(U)\to R^\times$. 
Since $R$ is a commutative ring, this homomorphism factors through $H_1(U,\Z) \to R^\times$. 
It is well known that $H_1(U,\Z)$ is a free $\Z$-module generated by $[\gamma_i]$, where $[\gamma_i]$ is the image of the meridian $\gamma_i$. 
Then $\sL$ determines an $n$-tuple $(q_1,\dots,q_n)\in (R^\times)^n$, where $q_i=\rho_\sL(\gamma_i)$. 
For each $X\in L(\sA)$, let $q_X=\prod_{H_i\supset X}q_i$ and $q_\infty=\prod_{i=1}^n q_i^{-1}$, which is the local monodromy around the hyperplane at infinity.

\begin{defn}[CDO-condition]
We say that $\sL$ satisfies the CDO-condition if $q_X\neq 1$ for each dense edge $X\subset \overline{H}_\infty$.
\end{defn}

The CDO-condition provides a purely combinatorial criterion for determining the local system cohomology of arrangement complements. 
The following theorem is proved by Cohen--Dimca--Orlik \cite{CDO03}. 
\begin{theorem}[\cite{CDO03}]
If a rank one $\C$-local system $\sL$ satisfies the CDO-condition, then
$$
H^k(U,\sL)=\left\{ 
\begin{aligned}
&\C^{|\chi(U)|} &(k=l), \\
&0   & (k\neq l),
\end{aligned}\right.$$
where $\chi(U)$ is Euler characteristic of $U$. 

\end{theorem}

For a rank one $\Z$-local system $\sL$, 
Sugawara \cite{Sug23} shows that 
if $\sA$ is an essential complexified real arrangement and $\sL$ satisfies CDO-condition, then 
$H^l(U,\sL)=\Z^{\beta_l}\oplus \Z_2^{\beta_{l-1}} $ and $H^k(U,\sL)=\Z_2^{\beta_{k-1}}$ for $1\leq k\leq l-1$, where $b_i=b_i(U)$ is the $i$-th Betti number of $U$ and $\beta_k=|\sum^k_{i=0}(-1)^i\cdot b_i|$. 
Furthermore, Liu, Maxim and Wang \cite{LMW24} generalized Sugawara's result to arbitrary essential complex arrangements.

Now we study the cohomology groups of finite cyclic covers of $U$.
Fix an epimorphism $\varphi:\pi_1(U) \to \Z_N$ defined by $ \gamma_i \mapsto \varepsilon_i$, where $N$ is a positive integer and  $\gamma_H$ denotes the meridian around the hyperplane $H\in \sA$.
Then $\varphi$ induces an $N$-fold cover $U^\varphi$ of $U$. 
Let $p:U^{\varphi} \to U$ be the covering map. 
Given a commutative ring $R'$ with identity, we have the following isomorphism by using Leray spectral sequence \cite{Dim04}: 
$$
H^k(U^\varphi,R')
\cong H^k(U^\varphi,\underline{R'}) 
\cong H^k(U,p_*\underline{R'}),
$$
where $\underline{R'}$ is the constant sheaf on $U^\varphi$. 
If we take $R=R'[\Z_N]\cong R'[t]/(t^N-1)$, then $p_*\underline{R'}$ is a rank one $R$-local system on $U$.  
In fact, the rank one $R$-local system $\sL=p_*\underline{R'}$  corresponds to the representation $\pi_1(U) \to R^\times,\gamma_i \mapsto t^{\varepsilon_i}$. 
Denoting $\varepsilon_X=\sum_{H_i\supset X}\varepsilon_i$, we obtain $q_X=t^{\varepsilon_X}$.

In this paper, we focus on the following combinatorial condition for this covering space $U^\varphi$, which is analogous to the CDO-condition. 
In the following discussion, we follow the convention that $\gcd(0,N)=N$.
\begin{defn}[CDO-condition for cyclic $N$-covers]
\label{CDO_condition_N}
We say that $\sL$ (or $\varphi$) satisfies the CDO-condition for cyclic $N$-covers if $\gcd(\varepsilon_X,N)= 1$ for each dense edge $X\subset \overline{H}_\infty$. 

\end{defn}

This condition has the following explanation for the specific case $R'=\C$ and $R=\C[t^{\pm}]/(t^N-1)$. 
In fact, the direct image sheaf  $\sL=p_*\underline{\C}$ decomposes into the direct sum of $N$ rank one $\C$-local systems on $U$. 
If $\sL$ satisfies the CDO-condition for cyclic $N$-covers, then every rank one non-constant $\C$-local system in this decomposition strictly satisfies the CDO-condition.

Consider a double covering space $U^\varphi$ of the complement of a complexified real arrangement, which is induced by a surjective map $\varphi:\pi_1(U)\to \Z_2\cong \{\pm 1\}$. 
Let $\sL=p_*\underline{\Z}$ be the rank one $R=\Z[\Z_2]$-local system.
Liu and Liu \cite[Corollary 1.7]{LL23} show that the integral cohomology group of $U^\varphi$ is torsion-free if the associated $\Z$-local system $\sL'$ satisfies the CDO-condition, where $\sL'$ corresponds to the representation $\varphi:\pi_1(U)\to \{\pm 1\} = \Z^\times$. 
In fact, $\sL'$ satisfies the CDO-condition if and only if $\sL$ satisfies the CDO-condition for cyclic $2$-covers. 
For the Milnor fiber of an essential complexified real line arrangement in $\CP^2$, Williams \cite{Wil13} gave a partial answer to Question \ref{Milnorfibertor}. 
He established an inequality for the first Betti number of the Milnor fiber. 
The author and Liu \cite{LX26} extended Williams's result to any essential complex line arrangements in $\CP^2$.

\begin{theorem}[{\cite[Theorem 1.3]{LX26}}] 
\label{line_arrangement_case}
Let $\sA$ be an essential complex arrangement of $n+1$ lines in $\CP^2$ with complement $U$. 
Fix any line $H\in \sA$. Let $\{P_1,\dots,P_s\}$ denote the set of multiple points on $H$ and let $m_i$ denote  the  multiplicity of $P_i$ for $1\leq i\leq s$.  
Here since $\sA$ is essential, we have $s>1$. 
Fix an epimorphism $\varphi:\pi_1(U)\twoheadrightarrow \Z$ with $\varphi(\alpha_H)=\epsilon\in \Z$ and $\varepsilon_i=\sum_{H'\supset P_i}\varphi(\gamma_{H'})$, where $H'\in \sA$ and $\gamma_{H'}$ is the meridian of $H'$.

Consider the composed map $\pi_1(U)\overset{\varphi}{ \twoheadrightarrow} \Z \twoheadrightarrow \Z_N$ and let $U^{\varphi, N}$ denote the associated $N$-fold covering.
If  $\epsilon=1$ and $\varepsilon_i\neq 0$ for all $m_i>2$, then 
for any field $\K$ we have
$$
 b_1 (U^{\varphi, N},\K) \leq n+\sum^s_{i=1}(m_i-2)(\gcd(\varepsilon_i,N)-1) .
$$
Furthermore, if  $\gcd(\varepsilon_i,N)=1$ for all $m_i>2$, then $H_1(U^{\varphi, N},\Z)\cong \Z^n$ is torsion-free. 
\end{theorem}

The above condition ($\epsilon=1$ and $\gcd(\varepsilon_i,N)=1$ for all $m_i>2$) is a special case of CDO-condition for cyclic $N$-covers. 
In this paper, we prove the following theorem for complexified real arrangement. 
This result can be considered as a partial generalization of Theorem \ref{line_arrangement_case}. 
For more details and notation, see Section 2 and 3.

\begin{theorem}[Main Theorem]
\label{main_theorem}
Let $\sA$ be an essential complexified real arrangement in $\C^l$ with the complement $U$ and $\overline{H}_\infty$ be the hyperplane at infinity. 
Let $\overline{\sA}=\sA \cup \{\overline{H}_\infty\}$ be the projective arrangement  associated to $\sA$. 
Fix an epimorphism $\varphi : H_1(U,\Z) \twoheadrightarrow \Z_N$ with $\varphi(\gamma_H)=\varepsilon_H$, where $N$ is a positive integer and $\gamma_H$ is the meridian associated with $H\in \sA$. 
Let $\varepsilon_X=\sum_{H\supset X}\varepsilon_H$, where $X\in L(\overline{\sA})$ and $L(\overline{\sA}) $ is the lattice of $\overline{\sA}$. 
Let $b_i=b_i(U,\Z)$ be the $i$-th Betti number of $U$. 
Fix a generic flag $\sF$ which is near to $\overline{H}_\infty$.

Let $U^{\varphi}$ be the associated $N$-fold covering of $U$ induced by $\varphi$. 
Then,
\begin{equation*}
b_i(U^{\varphi},\K)\leq \left\{
\begin{aligned}
&N\cdot(-1)^l\chi(U)+\sum_{C\in \uch^{l}(\sA)}\gcd(\varepsilon_{X(C)},N) &(i=l), \\
&\sum_{C\in \ch^i(\sA)}\gcd(\varepsilon_{X(C)},N) & (0\leq i\leq l-1).
\end{aligned}
\right.
\end{equation*}
where $\chi(U)$ is the Euler characteristic and with the convention that $\gcd(0,N)=N$. 
Furthermore, if the map $\varphi$ satisfies $\gcd(\varepsilon_X,N)=1$ for each dense edge $X\subset \overline{H}_\infty$ (i.e. satisfies the CDO-condition for cyclic $N$-covers),  
then 
\begin{equation*}
H^i(U^{\varphi},\Z)=\left\{
\begin{aligned}
&\Z^{b_l+(N-1)(-1)^l\chi(U)} &(i=l), \\
&\Z^{b_i} & (0\leq i\leq l-1), \\
&0 & (\text{otherwise}), 
\end{aligned}
\right.
\end{equation*}

\end{theorem}

Using the main theorem, we obtain the following corollaries.  
\begin{cor}
Let $\mathcal{A}$ be an essential complexified real arrangement in $\mathbb{CP}^l$ with $p$ hyperplanes, where $p$ is a prime number, then the cohomology groups of Milnor fiber $F$ of $\sA$ are
\begin{equation*}
H^k(F,\Z)=\left\{
\begin{aligned}
&\Z^{b_l+(p-1)(-1)^l\chi(U)} &(k=l), \\
&\Z^{b_k} & (0\leq k\leq l-1), \\
&0 & (\text{otherwise}). 
\end{aligned}
\right.
\end{equation*}
This follows from the main theorem, since the Milnor fiber corresponds to the case where $\varepsilon_H=1$ for every hyperplane $H \in \sA$.  
\end{cor}

\begin{cor}
If there exists a hyperplane $H$ in the essential complexified real arrangement $\sA$ in $\CP^l$ such that $\gcd(\varepsilon_H,N)=1$ and  $H$ is transverse to all other hyperplanes, 
then the cohomology groups of the $N$-fold cyclic covering space $U^{\varphi}$ are
\begin{equation*}
H^k(U^{\varphi},\Z)=\left\{
\begin{aligned}
&\Z^{b_l+(N-1)(-1)^l\chi(U)} &(k=l), \\
&\Z^{b_k} & (0\leq k\leq l-1), \\
&0 & (\text{otherwise}), 
\end{aligned}
\right.
\end{equation*}
\end{cor}

\begin{conjecture}

The torsion-freeness result established in Theorem \ref{main_theorem} can be extended to any essential complex arrangements, provided that the associated local system satisfies the CDO-condition for cyclic $N$-covers.
\end{conjecture}

The paper is organized as follows. 
In section 2, we recall the construction of the chamber cochain complex.
In section 3 we give the proof of the main theorem.

\subsection*{Acknowledgments} The author would like to thank Yongqiang Liu for giving an idea and comments on this paper.

\section{Preliminary}

\subsection{Chambers and flags}

Let $\sA=\{H_1,\dots,H_n\}$ be a hyperplane arrangement in $\R^l$ and $\overline{H}_\infty$ be the hyperplane at infinity.   
For an affine arrangement $\sA$, we define the associated projective arrangement $\overline{\sA}=\{H_1,\dots,H_n,\overline{H}_\infty\}$.
Let $\sA^\C$ be the complexified real arrangement associated to the real arrangement $\sA$ and $U$ be the complement of $\sA^\C$. 
The intersection lattice $L(\sA)$ of $\sA$ is the poset consisting of nonempty intersections of hyperplanes in $\sA$. 
For each $X\in L(\sA)$, define the subarrangement $\sA_X=\{H\supset X \mid H\subset \sA\}$. 
The intersection $X\in L(\sA)$ is called a dense edge if $\sA_X$ is indecomposable, see \cite{Dim17}. 
The intersection lattice and dense edge of $\overline{\sA}$ are defined similarly.

A connected component of the complement $\R^l\backslash \bigcup^n_{i=1}H_i$ is called a chamber. 
Denote the set of all chambers, bounded chambers and unbounded chambers by $\ch(\sA),\bch(\sA)$, $\uch(\sA)$, respectively. Clearly, we have $\ch(\sA)=\bch(\sA) \sqcup \uch(\sA)$. 
For more details on chambers, we refer the reader to the papers by Yoshinaga \cite{Yos07,Yos12,Yos26}. 

\begin{defn}
A generic flag $\sF$ in $\R^l$ is a sequence of affine subspaces in $\R^l$
$$\sF:\emptyset=\sF^{-1} \subset \sF^0 \subset \sF^1 \subset \cdots \subset \sF^l=\R^l,$$
where each $\sF^k$ is a $k$-dimensional affine subspace such that $\dim(X\cap \sF^k)=k+\dim X-l$ for each $X\in L(\sA)$ with $k+\dim X>l$, and $X\cap \sF^k=\emptyset$ otherwise.
\end{defn}

\begin{defn}
We say that the flag $\sF$ is near to $\overline{H}_\infty$ if $\sF^{k-1}$ does not separate $0$-dimensional edges of $\sF^{k}\cap \sA$ for all $0\leq k \leq l$, where $\sA\cap\sF^k=\{H\cap \sF^k\mid H\in\sA\}$. 
\end{defn}

In the following discussion, we assume that a generic flag $\sF$ is near to $\overline{H}_\infty$.

\begin{defn}
For $k=0,\dots,l$, define 
$$\begin{aligned}
\ch^k(\sA)&=\{C\in \ch(\sA)\mid C\cap \sF^k \neq \emptyset, C\cap \sF^{k-1}=\emptyset \},\\
\bch^k(\sA)&=\{C\in \ch^k(\sA)\mid C\cap \sF^k  \text{ is bounded} \},\\
\uch^k(\sA)&=\{C\in \ch^k(\sA)\mid C\cap \sF^k  \text{ is unbounded} \}.
\end{aligned}$$
\end{defn}

Clearly we have 
$$\begin{aligned}
\ch^k(\sA)&=\bch^k(\sA)\sqcup \uch^k(\sA), \\
\ch(\sA)&=\bigsqcup_{k=0}^l \ch^k(\sA).
\end{aligned}$$
Since $\sF$ is near to $\overline{H}_\infty$, every chamber $C\in \bch^k(\sA)$ is an unbounded chamber for $k<l$, that is, $\bch^k(\sA) \subset \uch(\sA)$ for $k<l$. 
For $k=l$, we have $\bch(\sA)=\bch^l(\sA)$.

\begin{defn}
\begin{enumerate}
\item 
For each $C\in \uch(\sA)$, let $\overline{C}$ be the closure of $C$ in the projective space $\mathbb{RP}^l$. 
Let $X(C)$ be the minimum subspace containing $\overline{C}\cap\overline{H}_\infty$. 

\item 
For each $C\in\uch(\sA)$, there exists a unique chamber which is the opposite chamber with respect to $\overline{C}\cap\overline{H}_\infty$. 
Let $C^\vee$ be the opposite chamber of $C$. 
Obviously we have $C^{\vee \vee}=C$. 

\end{enumerate}
\end{defn}

\begin{defn}
Define the involution $\iota$ by
\begin{align*}
    \iota:\uch(\sA) &\longrightarrow \uch(\sA)\\
    C &\longmapsto C^\vee.
\end{align*}
\end{defn}

Note that $\iota$ is a bijection. Restricting this map to $\bch^k(\sA)$, we have the following proposition. 
\begin{prop}[{\cite[Theorem 2.10]{Yos12}}]
\label{bch_uch_bij}
If $\sA$ is essential, 
then the involution $\iota$ induces a bijection 
$$\iota:\bch^{k}(\sA) \xrightarrow{\cong} \uch^{k+1}(\sA), $$
for $0\leq k \leq l-1$. Thus it follows that $\#\bch^k(\sA)=\#\uch^{k+1}(\sA)$.
\end{prop}

\begin{remark}
The condition ($\sA$ is essential) is necessary. 
Consider the arrangement $\sA=\{H_1=\{x=0\},H_2=\{x-1=0\}\}$ in $\R^2$, then the middle chamber $C$ satisfies that $C=C^\vee$ and hence $\iota$ is not a bijection between $\bch^k(\sA)$ and $ \uch^{k+1}(\sA)$.
\end{remark}

\begin{prop}[{\cite[ Proposition 2.4]{Yos12}}]
\label{chamber_dense_edge}
Let $X\in L(\overline{\sA})$ be an intersection contained in $\overline{H}_\infty$.
Then $X$ is a dense edge if and only if there exists a chamber $C\in \uch(\sA)$ such that $X=X(C)$.
\end{prop}

Denote by $b_k=b_k(U)$ the $k$-th Betti number and let $\beta_k=\#\bch^k(\sA)$. 

\begin{prop}[{\cite[Proposition 2.3.2]{Yos07},\cite[Proposition 2.6]{Sug23}}]
\label{betti_number}
If $\sA$ is essential, then for $0\leq k \leq l$, we have 
\begin{enumerate}
\item 
$b_k=\#\ch^k(\sA).$
\item 
$\beta_k= (-1)^k\sum^k_{i=0}(-1)^i b_i $. 
In particular, $\beta_l=\bch^l(\sA)=\bch(\sA)=(-1)^l\chi(U)$. 
\end{enumerate}
\end{prop}

\begin{defn}
For $C,C'\in \ch(\sA)$, let $\sep(C,C')=\{H\in \sA \mid H \text{ separates } C \text{ and } C' \}$ be the set of separating hyperplanes of $C,C'$.
\end{defn}

\begin{prop}[{\cite[Proposition 2.3]{BY16}}]
\label{sep_and_xc}
Let $C\in \uch(\sA)$. Then 
$$\sep(C,C^\vee)=\{H\in \sA \mid H\nsupseteq X(C) \}=\overline{\sA}\setminus \overline{\sA}_{X(C)}.$$
In particular, if $\dim X(C)=l-1$, i.e., $X(C)=\overline{H}_\infty$, then $\sep(C,C^\vee)=\sA$. 
\end{prop}

\begin{ex}
The figure \ref{chamber fig} shows chambers of a real hyperplane arrangement with $5$ lines and a hyperplane at infinity.

\begin{figure}[htpb]
\centering
\begin{tikzpicture}
    
\draw [thick, densely dashed](-0.5,1)--(13,1);
\fill[black] (0.2,1) circle (0.06)  node[above] {$\sF_0$};
\draw (13,1) node[right]{$\sF_1$};

\draw[thick] (0,0.44)--(12,7.16);
\draw (0.2,0.5) node[below] {$H_1$};
\draw[thick] (4,0.5)--(4,4) to [out=90,in=225] (7.5,7.7);
\draw (4,0.5) node[below] {$H_2$};
\draw[thick] (6,0.5)--(6,7.7);
\draw (6,0.5) node[below] {$H_3$};
\draw[thick] (8,0.5)--(8,4) to [out=90,in=315] (4.5,7.7);
\draw (8,0.5) node[below] {$H_4$};
\draw[thick] (0,7.16)--(12,0.44);
\draw (12,0.5) node[below] {$H_5$};
\draw[thick] (0,6.6)--(9,6.6) to [out=0,in=90](12.5,4) -- (12.5,0.5);
\draw (12.8,0.5) node[below] {$\overline{H}_\infty$};

\draw (2,3.8) node{$C_0$};
\draw (3,1.5) node{$C_1$};
\draw (5,1.5) node{$C_2$};
\draw (7,1.5) node{$C_3$};
\draw (9,1.5) node{$C_4$};
\draw (10,3.8) node{$C_0^{\vee}$};
\draw (8.2,5.8) node{$C_1^{\vee}$};
\draw (5.3,5) node{$C_2^{\vee}$};
\draw (6.7,5) node{$C_3^{\vee}$};
\draw (3.8,5.8) node{$C_4^{\vee}$};
\draw (5.3,7.5) node{$C_3$};
\draw (6.7,7.5) node{$C_2$};
\draw (5,3.8) node{$D_1$};
\draw (7,3.8) node{$D_2$};

\fill[black] (6,6.6) circle(0.08) ;
\draw [thick,->] (4.5,7)--(5.5,6.7);
\draw (4.3,7.2) node[left]{$X(C_2)=X(C_3)$};

\draw [thick] (6.3,6.7)--(9.1,6.7) to [out=0,in=155](10.6,6.45);
\draw [thick,->] (8.5,7.2)--(8,6.9);
\draw (9.5,7.2) node{$\overline{C}_1^{\vee}\cap \overline{H}_\infty$};

\end{tikzpicture}
\caption{A flag and chambers}
\label{chamber fig}
\end{figure}
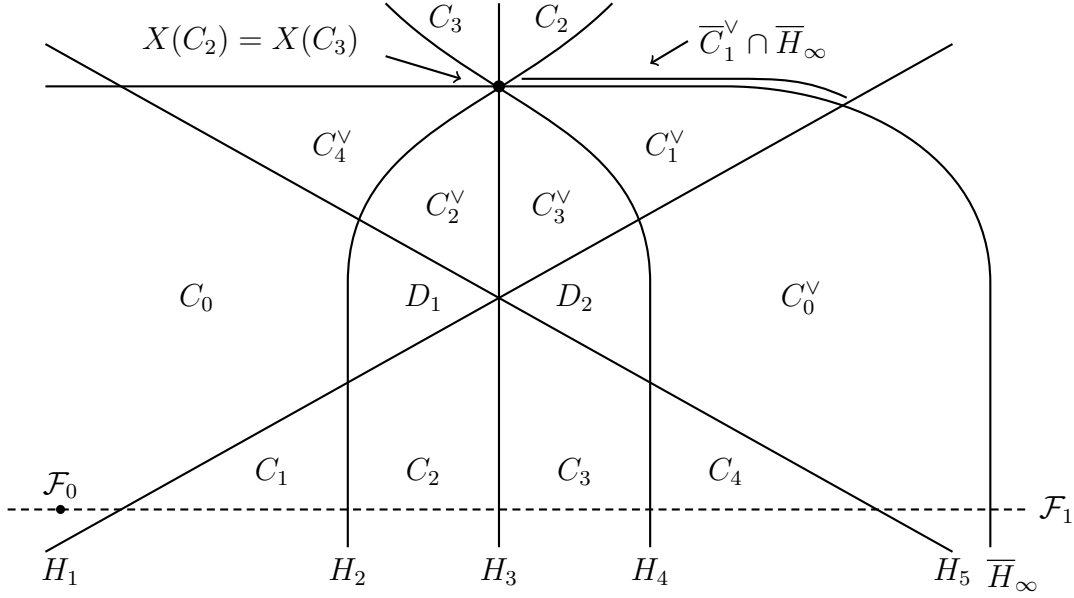

\end{ex}

\subsection{Deligne groupoid}

In this subsection, we briefly introduce the Deligne groupoid $\gal(\sA)$ and related theorems.
For details, see \cite{Yos07,Yos26}.

A path of length $k$ is a sequence $(C_0,C_1,\dots,C_k)$ of chambers such that the chambers $C_{i-1}$ and $C_i$ are adjacent. 
If $(C_0,C_1,\dots,C_k),(D_0,D_1,\dots,D_l)$ are two paths with $C_k=D_0$, then their composition is defined as
$$(C_0,C_1,\dots,C_k=D_0,D_1,\dots,D_l).$$

Now consider signed paths
$$(C_0*_{\sigma_1}C_1*_{\sigma_2}\dots*_{\sigma_k}C_k),$$
where $\sigma_i\in \{+,- \}$. 
If all signs are $+$, we identify $(C_0*_+C_1*_+\dots*_+C_k)$ with the path $(C_0,C_1,\dots,C_k)$;
if all signs are $-$, we identify $(C_0*_-C_1*_-\dots*_-C_k)$ with the path $(C_k,C_{k-1},\dots,C_1)^{-1}$.

\begin{defn}
Define $\gal(\sA)$ to be a category whose set of objects is the set of all chambers of $\sA$, i.e., $\mathrm{obj}(\gal(\sA))=\ch(\sA)$.  
For objects $C,C'$ in this category, the set of morphisms is defined as
$$\Hom(C,C')=\{ (C_0*_{\sigma_1}C_1*_{\sigma_2}\dots*_{\sigma_k}C_k) \mid k\geq 0,C_i\in\ch(\sA),C_0=C,C_k=C' \}/\sim,$$
where the equivalence relation $\sim$ is described in \cite{Yos26}.
\end{defn}

\begin{theorem}[\cite{Yos26}]
For any chamber $C\in \ch(\sA)$, we have the isomorphism 
$$\Hom(C,C)\cong \pi_1(U,p_C),$$
where $p_C$ is an arbitrary point in $C$.
\end{theorem}

Furthermore, the category of representations of the Deligne groupoid $\gal(\sA)$ is equivalent to the the category of representations of the fundamental group and the category of local systems on $U$.

\subsection{Chamber cochain complex}

Now we describe how to compute the cohomology of a rank one local system. 
Fix a generic flag $\sF$ on $\sA$, which is near to $\overline{H}_\infty$. 
We will describe the twisted cochain complex 
$$(R[\ch^\bullet(\sA)],d'_\sL) : \cdots \longrightarrow R[\ch^k(\sA)] \stackrel{d'^k_\sL}{\longrightarrow} R[\ch^{k+1}(\sA)]\longrightarrow \cdots$$
which computes local system cohomology for an arbitrary rank one $R$-local system $\sL$, that is, 
$$H^*(R[\ch^\bullet(\sA)],d'_\sL)\cong H^*(U,\sL).$$

To describe the boundary map $d'_\sL$, we need to define two maps on the chambers of $\sA$: 
$$\deg:\ch^k(\sA)\times \ch^{k+1}(\sA) \rightarrow \Z,$$ 
and 
$$\Delta': \ch(\sA)\times \ch(\sA) \to R .$$

\subsubsection{The deg map}

Now we define the deg map, for more details, see \cite{Yos07}. 
Let $D=D^k \subset \sF^k$ be a sufficiently large $k$-dimensional ball so that every $0$-dimensional edge $X$ of $\sA\cap \sF^k$ is in the interior of $D^k$. 
Let $C\in \ch^k(\sA), C'\in \ch^{k+1}(\sA)$. 
There exists a tangent vector field $V^{C'}$ on $\sF^k$ that satisfies the following conditions. 
\begin{enumerate}
\item $V^{C'}(x) \neq 0$ for $x\in \partial(\overline{C}\cap D)$. 
\item $V^{C'}(x)$ points toward the interior of $D$ for $x\in \partial D$.
\item If $x\in H \cap\sF^k$ for $H\in\sA$, then $V^{C'}(x)\notin T_x(H\cap\sF^k)$ and $V^{C'}(x)$ points toward the side containing $C'$, where $T_x(H\cap\sF^k)$ is the tangent space of $H\cap\sF^k$ at $x$. 
\end{enumerate}

Now we define the deg map on chamber of $\sA$. 
\begin{defn}
Let $C\in \ch^k(\sA), C'\in \ch^{k+1}(\sA)$ and fix $V^{C'}$. 
Define $\deg(C,C')$ as follows. 

\begin{enumerate}
\item When $k=0$, $\deg(C,C')=1$. 
\item When $k \geq 1$, 
$$\deg(C,C'):=\deg \left( 
\frac{V^{C'}}{|V^{C'}|} :  \partial(\overline{C}\cap D) \rightarrow S^{k-1} 
\right) \in \Z,$$
where $S^{k-1}$ is a $(k-1)$-dimensional sphere. 
\end{enumerate}
\end{defn}

For some special cases, there is a classical result of deg map. 
\begin{prop}[{\cite[Theorem 4.8]{BY16},\cite[Section 3.2]{Yos12}}]
\label{deg_chamber}
\begin{enumerate}
\item[] 
\item For $0\leq k \leq l-1$, let $C\in \bch^k(\sA)$ . 
Then $$\deg(C,C^\vee)=(-1)^{l-1-\dim X(C)}.$$
\item  Let $C\in \bch^1(\sA)$. Then $C\cap \sF^1$ is a closed interval, whose boundary consists of two points, which can be expressed as $(H\cap \sF^1)\cup (H'\cap \sF^1)$ for $H,H'\in \sA$. Then $\deg(C,C')$ can be computed as 
$$ \deg(C,C')=\left\{
\begin{aligned}
1  & \text{ if } H,H'\in\sep(C,C'),\\
-1 & \text{ if } H,H'\notin\sep(C,C'),\\
0  & \text{ otherwise }. 
\end{aligned}
\right. $$

\item Let $C\in \uch^1(\sA)$. Then $C\cap \sF^1$ is an unbounded interval, whose boundary (a point) can be expressed as $(H\cap \sF^1)$. Then $\deg(C,C')$ can be computed as 
$$\deg(C,C')=\left\{
\begin{aligned}
-1 & \text{ if } H\notin\sep(C,C'),\\
0  & \text{ if } H\in\sep(C,C').
\end{aligned}
\right. $$

\end{enumerate}
\end{prop}

\begin{prop}[{\cite[Theorem 4.6]{BY16}}]
\label{up_triangular}
For $0 \leq k \leq l-1$, let us fix an ordering of chambers $\bch^k(\sA) = \{C_1,\dots,C_b\}$ such that
$$\dim X(C_1)\geq \dim X(C_2) \geq \cdots \geq \dim X(C_b).$$
Then the matrix $(\deg(C_i,C_j^\vee))_{i,j}$ is upper triangular, that is, if $i>j$, then $\deg(C_i,C_j^\vee)=0$. 
\end{prop}

\subsubsection{The map $\Delta$}

To define the map $\Delta'$, we need to define the map $\Delta(C,C')\in \widetilde{R}$ for chambers $C,C'\in \ch(\sA)$, where $\widetilde{R}$ is the extended ring of $R$.  
For more details, see\cite[Section 3.3, Section 6.4]{Yos07} or \cite[Section 7.1]{Yos26}. 

\begin{defn} 
For a commutative ring with identity, define the extended ring as follows
$$\widetilde{R}:=R[x_a \mid a\in R^\times]/(x^2_a-a;a\in R^\times).$$
\end{defn}

\begin{prop}[{\cite[Lemma 7.1]{Yos26}}]
\begin{enumerate}
\item[] 
\item The natural map $i:R\rightarrow \widetilde{R}$ is injective.
\item $x_a\in \widetilde{R}^\times$.
\end{enumerate}
\end{prop}

Recall that $\rho:\pi_1(U)\to R^\times, \gamma_i \mapsto q_i$, where $\gamma_i$ are the meridians of $H_i$.  
Now we construct the representations of the Deligne groupoid $\gal(\sA)$. 
Let $C\in \ch(\sA)$. Define the $R$-submodule $\rho(C) \subset \widetilde{R}$ by 
$$\rho(C):=R \cdot \prod_{H_i\in \sep(C_0,C)} q^{1/2}_i,$$
where $C_0$ is the chamber containing $\sF^0$ and $q^{1/2}_i$ denotes $x_{q_i}$.

\begin{prop}[{\cite[Proposition 7.2]{Yos26}}]

Let $\mathrm{Mod}_R$ be the category of $R$-modules. 
The correspondence 
$$
\begin{aligned}
\rho:\ch(\sA) &\longrightarrow \mathrm{Mod}_R \\
C&\longmapsto \rho(C) \\
(C,C') &\longmapsto x_{q_i},
\end{aligned}
$$
where $C$ and $C'$ are adjacent and separated by $H_i$ and $ (C,C')\in \Hom(C,C')$, gives a representation of the groupoid $\mathrm{Gal}(\sA)$. 
\end{prop}

In the following discussion, we denote $\prod_{H_i\in \sep(C_0,C)}q_i$ by $q_{\sep(C_0,C)}$. 

\begin{defn}
Let $C,C'\in\ch(\sA)$. Define the map $\Delta$ by 
$$\Delta(C,C')=q^{1/2}_{\sep(C,C')}-q^{-1/2}_{\sep(C,C')}\in \widetilde{R}.$$
\end{defn}

\subsubsection{The map $\Delta'$ and the chamber cochain complex}

Using the above notions, we have the following result proved by Yoshinaga. 
\begin{theorem}[{\cite[Theorem 6.4.1]{Yos07},\cite[Proposition 7.5]{Yos26}}]
\label{local_cohomology}
Let $\rho:\pi_1(U)\rightarrow R^\times$ be the representation  corresponding to a rank one $R$-local system $\sL$.
Let
$$\sC^k(\sA,\rho):=\bigoplus_{C\in \ch^k(\sA)}\rho(C).$$
Define the map $d_\sL:\sC^k(\sA,\rho)\rightarrow \sC^{k+1}(\sA,\rho)$ by 
$$d_\sL(y)=
\sum_{C'\in\ch^{k+1}(\sA)} \deg(C,C') \cdot \Delta(C,C')\cdot y,
$$
where $y\in\rho(C)$. Then $(\sC^\bullet(\sA,\rho),d_\sL)$ gives a cochain complex such that  
$$H^k(\sC^\bullet(\sA,\rho),d_\sL)\cong H^k(U,\sL_\rho) $$
for any $k$ .
\end{theorem}

Now we modify the cochain complex $(\sC^\bullet(\sA,\rho),d_\sL)$. 
Note that if we take $y=q^{1/2}_{\sep(C_0,C)}=\prod_{H_i\in \sep(C_0,C)} q^{1/2}_i$, then 
$$\begin{aligned}
\Delta(C,C')\cdot y
&=\left( q^{1/2}_{\sep(C,C')}-q^{-1/2}_{\sep(C,C')} \right) 
\cdot q^{1/2}_{\sep(C_0,C)}\\
&=\left( 1-q^{-1}_{\sep(C,C')} \right) \cdot q^{1/2}_{\sep(C,C')}\cdot q^{1/2}_{\sep(C_0,C)} \cdot q^{-1/2}_{\sep(C_0,C')} \cdot q^{1/2}_{\sep(C_0,C')}
\end{aligned}$$
 
Note that $ q^{1/2}_{\sep(C,C')}\cdot q^{1/2}_{\sep(C_0,C)} \cdot q^{-1/2}_{\sep(C_0,C')}\in R^\times$. 
In fact, 
$$\begin{aligned}
& q^{1/2}_{\sep(C,C')}\cdot q^{1/2}_{\sep(C_0,C)} \cdot q^{-1/2}_{\sep(C_0,C')} \\
=&q^{1/2}_{\sep(C_0,C)}\cdot q^{1/2}_{\sep(C,C')}\cdot q^{-1/2}_{\sep(C_0,C')} \\
\in& \rho ( \Hom(C',C_0)\circ \Hom(C,C') \circ \Hom(C_0,C)) \\
\subset &\rho(\Hom(C_0,C_0))=\rho(\pi_1(U))\subset R^\times.
\end{aligned}$$
Furthermore, it is easy to see that $q^{1/2}_{\sep(C,C')}\cdot q^{1/2}_{\sep(C_0,C)} \cdot q^{-1/2}_{\sep(C_0,C')}=q_{\sep(C,C') \cap \sep(C_0,C)}$.

The above result shows that 
$\Delta(C,C')\cdot q^{1/2}_{\sep(C_0,C)} \cdot q^{-1/2}_{\sep(C_0,C')}  \in R$. 
Let 
$$\Delta'(C,C')=\Delta(C,C')\cdot q^{1/2}_{\sep(C_0,C)} \cdot q^{-1/2}_{\sep(C_0,C')}.$$
Then for any chambers $C,C'$, we have $\Delta'(C,C')\in R$.

If we denote $q^{1/2}_{\sep(C_0,C)}$ by a formal symbol $[C]$, then $\rho(C)=R\cdot[C]$ and 
$$\sC^k(\sA,\rho)=\bigoplus_{C\in \ch^k(\sA)}R\cdot[C].$$

Let 
$$
d'^k_\sL([C]) =\sum_{C'\in \ch^{k+1}(\sA)} \deg(C,C')\cdot  \Delta'(C,C')\cdot [C'].
$$
Then $d^k_\sL([C])=d^k_\sL(q^{1/2}_{\sep(C_0,C)})=d'^k_\sL([C])$.

By Theorem \ref{local_cohomology}, we obtain the following corollary.  
\begin{cor}\label{local_cohomology_cor}
With the above notations, $(R[\ch^\bullet(\sA)],d'_\sL)$ is a cochain complex.
Furthermore, 
$$H^k(R[\ch^\bullet(\sA)],d'_\sL)\cong H^k(U,\sL)$$
as $R$-modules for any $k$.
\end{cor}

\begin{remark}
If $q_i^{1/2}\in R$ for each $i$, then $\rho(C)=R$ for each $C\in \ch^\bullet(\sA)$.
Then we have another way to construct cochain complex.  
Let $C$ be a symbol, $\rho(C)=R \cdot [C]$ 
and 
$$d^k_\sL([C])=\sum_{C'\in \ch^{k+1}(\sA)}  \deg(C,C')\cdot  \Delta(C,C')\cdot [C'].$$
Then we obtain that $(R[\ch^\bullet(\sA)],d_\sL)$ is a cochain complex and  $H^k(R[\ch^\bullet(\sA)],d_\sL)\cong H^k(U,\sL)$ for any $k$.
As an example, one can see \cite[Theorem 6.4.1]{Yos07} for the case $R=\C$.
Recently, the case $R=\Z$ was considered in \cite{Sug23}. 
\end{remark}

\begin{ex}
In the setting of Example \ref{chamber fig}, 
for the boundary map $d'_\sL: R[\ch^1(\sA)] \to R[\ch^2(\sA)]$, we can compute its corresponding matrix as following, 
\begin{equation*}
\begin{pNiceArray}{cccccc}[first-row,first-col]
 & C_1^\vee & C_4^\vee & C_2^\vee & C_3^\vee & D_1 & D_2 \\
C_1 & (1-q_{12345}^{-1})  q_1 & 0 & (1-q_{125}^{-1})  q_1 & (1-q_{1235}^{-1})  q_1 & (1-q_{12}^{-1})  q_1 & 0  \\
C_4 & 0 & (1-q_{12345}^{-1})  q_{1234} & (1-q_{1345}^{-1})  q_{134} & (1-q_{145}^{-1})  q_{4} & 0 & (1-q_{45}^{-1})  q_{4}  \\
C_2 & 0 & 0 & -(1-q_{15}^{-1})  q_{1} & 0 & -(1-q_{1}^{-1})  q_{1} & 0  \\
C_3 & 0 & 0 & 0 & -(1-q_{15}^{-1})  q_{1} & 0 & -(1-q_{5}^{-1})  \\
C_0^\vee & -(1-q_{1}^{-1})  q_{1} & -(1-q_{1234}^{-1})  q_{1234} & -(1-q_{134}^{-1})  q_{134} & -(1-q_{14}^{-1})  q_{14} & 0 & -(1-q_{4}^{-1})  q_{4}  \\
 
\end{pNiceArray}.
\end{equation*}

\end{ex}

\section{Main results}

In this section, we provide the proof of main theorem of this paper. 
To prove the main theorem, we need the following lemma. 
\begin{lemma}[{\cite[Lemma 3.6]{LX26}}]
\label{TN}
Consider the integer matrix $T_N$ of size $N\times N$:
$$T_N=
\begin{pNiceMatrix}
0 & 1 & 0 & \cdots & 0\\
0 & 0 & 1 & \cdots & 0\\
\vdots & & & & \vdots\\
0 & 0 & 0 & \cdots & 1\\
1 & 0 & 0 & \cdots & 0\\
\end{pNiceMatrix}.
$$ 
For any integer $k$ and any field $\K$, $\rk_\K \big(T_N^k-I_N\big)=N-\gcd(k,N)$, where $I_N$ is the identity matrix and with the convention that $\gcd(0,N)=N$. 

\end{lemma}

\begin{proof}[Proof of main theorem]

Let $R=\Z[\Z_N]=\Z[t^\pm]/(t^N-1)$. 
Then the map $\varphi$ corresponds to the following representation:
$$\begin{aligned}
\rho_{\varphi}:H_1(U,\Z)  &\longrightarrow \Z[\Z_N]^\times = R^\times \\
[\gamma_i] &\longmapsto  q_i:=t^{\varepsilon_i}.
\end{aligned}$$

Let $\sL$ be the $R$-local system associated to the map $\varphi$. 
Consider the $R$-module cochain complex $(R[\ch^\bullet(\sA)],d'_\sL)$. 
By Corollary \ref{local_cohomology_cor}, we have an isomorphism $H^k(R[\ch^\bullet(\sA)],d'_\sL)\cong H^k(U,\sL)$ as $R$-modules for any $k$ .

Recall that $\bch^{k}(\sA)=\uch^{k+1}(\sA)$ by Proposition \ref{bch_uch_bij}. 
Let $\beta_k=\#\bch^k(\sA)$ for $0 \leq k \leq l$. 
Now fix $k$ and fix an ordering of chambers $\bch^k(\sA) = \{C_1,\dots,C_{\beta_k}\}$ so that
$$\dim(X(C_1))\geq \dim(X(C_2)) \geq \dots \geq \dim(X(C_{\beta_k})).$$
Let $\uch^k(\sA)=\{B_1,\cdots,B_{\beta_{k-1}}\}$ and $\bch^{k+1}(\sA)=\{D_1,\dots,D_{\beta_{k+1}} \}$. 
Then we can fix an ordering of chambers 
$$\begin{aligned}
\ch^k(\sA)&=\{C_1,\dots,C_{\beta_k},B_1,\cdots,B_{\beta_{k-1}} \}, \\
\ch^{k+1}(\sA)&=\{C_1^\vee,\dots,C_{\beta_k}^\vee,D_1,\dots,D_{\beta_{k+1}} \}.
\end{aligned}
$$

Denote $\deg(C,C^\vee)\cdot \Delta'(C,C^\vee)$ by $a_C$.
Then we obtain a matrix associated to $d'^k_\sL$:
\begin{equation}
\begin{pNiceArray}{cccc:ccc}[first-row,first-col]
    &C_1^\vee&C_2^\vee & \cdots &C_{\beta_k}^\vee &D_1  &\cdots &D_{\beta_{k+1}} \\
C_1 &a_{C_1}&* & * &* &\Block{4-3}<\Large>{*} &  &  \\
C_2 &0 & a_{C_2}&* & * &  &  & \\
\vdots &\vdots &\vdots &\ddots &* &  &  & \\
C_{\beta_k} &0 &0 &0 &a_{C_{\beta_k}}  &  &  & \\
\hdottedline
B_1&\Block{3-4}<\Large>{*}& & & & \Block{3-3}<\Large>{*} \\
\vdots&\\
B_{\beta_{k-1}}&

\end{pNiceArray},
\end{equation}
where upper-left submatrix is upper triangular by Proposition \ref{up_triangular}.

Recall that $H^k(U,\sL)\cong H^k(U^\varphi,\Z)$ as $\Z$-modules for any $k$.
Similarly, we have an isomorphism of $\K$-modules $H^k(U,\sL\otimes\K)\cong H^k(U^\varphi,\K)$.
To compute the $\Z$-module $H^k(U,\sL)$, we need to replace $t$ by $T_N$ in the boundary map $d'_\sL$ and denote the new boundary map by $A_k$. 
Then we obtain a $\Z$-coefficient cochain complex (where $R$ is considered as $\Z[T_N]$): 
$$(R[\ch^\bullet(\sA)],A_*): \cdots \longrightarrow R[\ch^{k}(\sA)] \stackrel{A_k}{\longrightarrow} R[\ch^{k+1}(\sA)] \longrightarrow \cdots,$$
such that for any $k$, we have an isomorphism 
$H^k(R[\ch^\bullet(\sA)],A_*)\cong H^k(U^\varphi,\Z)$ as $\Z$-modules.
Similarly, we have an isomorphism $H^k(R\otimes\K[\ch^\bullet(\sA)],A_*\otimes\K)\cong H^k(U^\varphi,\K)$ as $\K$-modules.

Now we analyze the diagonal blocks of the upper-left submatrix of $A_k$.
First, by definition, we have
$$\Delta'(C,C')=\left( 1-q^{-1}_{\sep(C,C')} \right) \cdot q^{1/2}_{\sep(C,C')}\cdot q^{1/2}_{\sep(C_0,C)} \cdot q^{-1/2}_{\sep(C_0,C')}.$$
Note that the part $q^{1/2}_{\sep(C,C')}\cdot q^{1/2}_{\sep(C_0,C)} \cdot q^{-1/2}_{\sep(C_0,C')}\in R^\times$ of $\Delta'(C,C')$ is a power of $t$.
Then 
$$(q^{1/2}_{\sep(C,C')}\cdot q^{1/2}_{\sep(C_0,C)} \cdot q^{-1/2}_{\sep(C_0,C')})(T_N)$$ 
is an invertible matrix with every entry belongs to $\Z$ and its determinant is $\pm1$. 

By Proposition \ref{sep_and_xc}, we have $1-q^{-1}_{\sep(C,C^\vee)}=1-q_{X(C)}=1-t^{\varepsilon_{X(C)}}$. 
Combined with Lemma \ref{TN}, this implies that  
$$\rk_\K(1-q_{X(C)})(T_N)=N-\gcd(\varepsilon_{X(C)},N).$$
Recall from Proposition \ref{deg_chamber} that $\deg(C,C^\vee)=(-1)^{l-1-\dim X(C)} \neq 0$. 
Thus, we have:
$$\begin{aligned}
\rk_\K A_k
&\geq \sum_{C\in \bch^k(\sA)}\rk_\K(\deg(C,C^\vee)\cdot\Delta(C,C^\vee))(T_N) \\
&=\sum_{C\in \bch^k(\sA)}\rk_\K(1-q_{X(C)})(T_N) \\
&=\sum_{C\in \bch^k(\sA)}(N-\gcd(\varepsilon_{X(C)},N))\\
&=N\#\bch^k(\sA)-\sum_{C\in \bch^k(\sA)}\gcd(\varepsilon_{X(C)},N) 
\end{aligned}$$

Note that $X(C)=X(C^\vee)$ and $\#\bch^{k-1}(\sA)=\#\uch^{k}(\sA)$ for $0\leq k\leq l-1$, then we obtain the following inequality:
$$\begin{aligned}
\label{Hkinequ}
\dim_\K H^k(U^\varphi,\K)
&=\dim_\K \frac{\ker A_{k}}{\im A_{k-1}} \\
&=\dim_\K \ker A_k-\dim_\K \im A_{k-1} \\
&=\dim_\K R\otimes\K[\ch^k(\sA)]- \dim_\K \im A_k-\dim_\K \im A_{k-1} \\
&=N\#\ch^k(\sA) -\rk_\K A_k- \rk_\K A_{k-1} \\
&\leq N\#\ch^k(\sA) - N\#\bch^{k}(\sA) - N\#\bch^{k-1}(\sA) \\
&\quad\quad
+\sum_{C\in \bch^k(\sA)}\gcd(\varepsilon_{X(C)},N) 
+\sum_{C\in \bch^{k-1}(\sA)}\gcd(\varepsilon_{X(C)},N) \\
&=\sum_{C\in \ch^k(\sA)}\gcd(\varepsilon_{X(C)},N) .
\end{aligned}$$
 
Similarly for $k=l$, we have 
$$\begin{aligned}
\dim_\K H^l(U^\varphi,\K)
&=\dim_\K \frac{R\otimes\K[\ch^l(\sA)]}{\im A_{l-1}}\\
&=\dim_\K R\otimes\K[\ch^l(\sA)] - \dim_\K \im A_{l-1}\\
&\leq N\#\ch^l(\sA)-N\#\bch^{l-1}(\sA)+\sum_{C\in \bch^{l-1}(\sA)}\gcd(\varepsilon_{X(C)},N) \\
&=N\#\ch^l(\sA)-N\#\uch^{l}(\sA)+\sum_{C\in \uch^{l}(\sA)}\gcd(\varepsilon_{X(C)},N)\\
&=N\#\bch^{l}(\sA)+\sum_{C\in \uch^{l}(\sA)}\gcd(\varepsilon_{X(C)},N)\\
&=N\cdot (-1)^l \chi(U)+\sum_{C\in \uch^{l}(\sA)}\gcd(\varepsilon_{X(C)},N)\\
\end{aligned}$$
where $\#\bch^{l}(\sA)=(-1)^l \chi(U)$ by Proposition  \ref{betti_number}.
Then we prove the inequality in the theorem. 

Furthermore, if the map $\varphi$ satisfies that $\gcd(\varepsilon_X,N)=1$ for each dense edge $X\subset \overline{H}_\infty$, 
Then $\gcd(\varepsilon_{X(C)},N)=1$ for each unbounded chamber $C$ by Proposition \ref{chamber_dense_edge}. 
Then we obtain the following inequality for $0\leq k\leq l-1$:
$$\begin{aligned}
\rk_\K A_k(T_N)&\geq N\#\bch^k(\sA)-\sum_{C\in \bch^k(\sA)}\gcd(\varepsilon_{X(C)},N)\\
&=(N-1)\#\bch^k(\sA),\\
\dim_\K H^k(U^\varphi,\K) &\leq\sum_{C\in \ch^k(\sA)}\gcd(\varepsilon_{X(C)},N) \\
&=\#\ch^k(\sA)=b_k,\\
\dim_\K H^l(U^\varphi,\K) &\leq N\#\bch^{l}(\sA)+\sum_{C\in \uch^{l}(\sA)}\gcd(\varepsilon_{X(C)},N)\\
&=N\#\bch^{l}(\sA)+\#\uch^{l}(\sA)\\
&=\#\ch^{l}(\sA)+(N-1)\#\bch^{l}(\sA)\\
&=b_l+(N-1)(-1)^l\chi(U) .
\end{aligned}$$

For any field $\K$, by the Universal Coefficient Theorem and \cite[Section 3.4]{Suc24}, we have
$$\dim_\K H^k(U^\varphi,\K)\geq \dim_\C H^k(U^\varphi,\C) \geq \dim_\C H^k(U,\C) =b_k.$$ 
Then for $0\leq k\leq l-1$, by this inequality, we have $\dim_\K H^k(U^\varphi,\K)=b_k(U)$. 
For $k=l$, since $\chi(U^\varphi)=N\cdot \chi(U)$, we obtain
$$\begin{aligned}
(-1)^l\dim_\K H^l(U^\varphi,\K)&= N\cdot\chi(U)-\sum_{0\leq k\leq l-1}(-1)^k\dim_\K H^k(U^\varphi,\K)\\
&=N\cdot\chi(U)-\sum_{0\leq k\leq l-1}(-1)^k b_k\\
&=(N-1)\chi(U)+(-1)^l b_l.
\end{aligned}$$
Since the above formula is true for any field $\K$, hence \begin{equation*}
H^k(U^\varphi,\Z)=\left\{
\begin{aligned}
&\Z^{b_l+(N-1)(-1)^l\chi(U)} &(i=l), \\
&\Z^{b_k} & (0\leq k\leq l-1), \\
&0 & (\text{otherwise}), 
\end{aligned}
\right.
\end{equation*}
by the Universal Coefficient Theorem.

\end{proof}

\begin{remark}

For the case $l=2$ and $\sA$ has $n$ lines, note that $\#\ch^1(\sA)=n$.
Consider the associated arrangement $\overline{\sA}$ in $\CP^2$. 

If $X\in L_2(\overline{\sA})$ such that $X=X(C)$ for some $C\in \ch^1(\sA)$, then there exist two lines $H,H'\in\sA$ such that $H,H'$ are parallel and $H\cap\overline{C}\neq \emptyset, H'\cap\overline{C}\neq \emptyset$. 
In fact, we have 
$$\#\{C\in \ch(\sA)\mid X(C)=X\}=n_X-2,$$
where $n_X$ is the number of the lines which contains $X$. 

Note that all dense edges in $\overline{H}_\infty$ are either points or $\overline{H}_\infty$. 
Assume that $\overline{H}_\infty$ contains $s$ intersection points. 
Let $D(\overline{\sA})$ be the set of dense edges contained in $\overline{H}_\infty$. 
Then we have 
$$n=\sum^s_{\substack{X\in L_2(\overline{\sA})\\X\subset \overline{H}_\infty}} (n_X-1)
=s+\sum^s_{\substack{X\in L_2(\overline{\sA})\\X\subset \overline{H}_\infty}} (n_X-2)
=s+\sum_{X\in D(\overline{\sA})\backslash\{\overline{H}_\infty\}}(n_X-2)$$

Using the above two equalities, we obtain that $\#\{C\in \ch^1(\sA) \mid X(C)=\overline{H}_\infty\}=s$.
Then by inequality \ref{Hkinequ}, we obtain 
$$
\begin{aligned}  
\dim_\K(H^1(U,\sL))&\leq \sum_{C\in \ch^1(\sA)}\gcd(\varepsilon_{X(C)},N)\\
&=n+\sum_{C\in \ch^1(\sA)} (\gcd(\varepsilon_{X(C)},N)-1 )\\
&=n+\sum_{X\in D(\overline{\sA})\backslash\{\overline{H}_\infty\}}(n_X-2)(\gcd(\varepsilon_{X},N)-1 )
+\sum_{\substack{C\in \ch^1(\sA)\\ X(C)=\overline{H}_\infty}}
(\gcd(\varepsilon_{\overline{H}_\infty},N)-1 )\\
&=n+\sum_{X\in D(\overline{\sA})\backslash\{\overline{H}_\infty\}}(n_X-2)(\gcd(\varepsilon_{X},N)-1 )
+s(\gcd(\varepsilon_{\overline{H}_\infty},N)-1 )
.\end{aligned}
$$

If $\gcd(\varepsilon_{\overline{H}_\infty},N)=1$, 
this result agree with Theorem \ref{line_arrangement_case}. 
Similarly, if $\varphi$ satisfies that $\gcd(\varepsilon_X,N)=1$ for each dense edge $X\subset \overline{H}_\infty$, we also have $b_1(U^\varphi,\Z)=n$ and $H_1(U^\varphi,\Z)$ is torsion-free. 

\end{remark}

\end{document}